\documentclass[letterpaper,11pt]{amsart}

\usepackage[margin=1.2in]{geometry}
\usepackage{amsmath,amsthm,amssymb}
\usepackage{xspace,xcolor}
\usepackage[breaklinks,colorlinks,citecolor=teal,linkcolor=teal,urlcolor=teal,pagebackref,hyperindex]{hyperref}
\usepackage[alphabetic]{amsrefs}
\usepackage[all]{xy}
\usepackage{color}
\usepackage{url}
\usepackage{tikz-cd}

\theoremstyle{plain}
\newtheorem{thm}{Theorem}[section]

\newtheorem{prop}[thm]{Proposition}
\newtheorem{cor}[thm]{Corollary}

\theoremstyle{definition}

\newtheorem{eg}[thm]{Example}

\theoremstyle{remark}
\newtheorem{rmk}[thm]{Remark}

\def\Z{{\mathbf Z}}

\def\C{{\mathbf C}}

\def\cO{\mathcal{O}}

\def\.{\cdot}
\def\^{\widehat}

\def\({\left(}
\def\){\right)}

\newcommand{\llbracket}{[\negthinspace[}
\newcommand{\rrbracket}{]\negthinspace]}

\renewcommand{\and}{ \ \ \text{ and } \ \ }

\newcommand{\factor}[2]{\left. \raise 2pt\hbox{$#1$} \right/\hskip -2pt\raise -2pt\hbox{$#2$}}

\DeclareMathOperator{\codim} {codim}

\DeclareMathOperator{\Cont} {Cont}

\makeatletter
\@namedef{subjclassname@2020}{\textup{2020} Mathematics Subject Classification}
\makeatother

\usepackage{soul}

\begin{document}

\author[D.~Bath]{Daniel Bath}
\address{KU Leuven, Departement Wiskunde, Celestijnenlaan 200B, Leuven
3001, Belgium}
\email{dan.bath@kuleuven.be}

\author[M.~Musta\c{t}\u{a}]{Mircea Musta\c{t}\u{a}}

\address{Department of Mathematics, University of Michigan, 530 Church Street, Ann Arbor, MI 48109, USA}

\email{mmustata@umich.edu}

\thanks{The first author was supported by FWO grant \#1282226N. The second author was partially supported by NSF grant DMS-2301463 and by the Simons Collaboration grant \emph{Moduli of
Varieties}}

\subjclass[2020]{14B05, 14J17, 32S25}

\begin{abstract}
Given a closed subscheme $Z$ in a smooth variety $X$, defined by the maximal minors of an $s\times r$ matrix of regular functions, with $s\geq r$, we 
consider the corresponding incidence correspondence $W$ in $Y=X\times {\mathbf P}^{r-1}$, and relate the log canonical thresholds of
$(X,Z)$ and $(Y,W)$. In particular, when $r=s$, we show that ${\rm lct}(X,Z)=1$ if and only if ${\rm lct}(Y,W)=r$. Moreover, in this case, we show that $Z$ has rational singularities
if and only if $W$ has pure codimension $r$ in $Y$ and has rational singularities. As a consequence, we deduce that for a configuration hypersurface with a connected configuration matroid,
the corresponding configuration incidence variety has rational singularities.
\end{abstract}

\title[On singularities of determinantal hypersurfaces]{On singularities of determinantal hypersurfaces}

\maketitle

\section{Introduction}

Let $X$ be a smooth, irreducible, complex algebraic variety, and let $A=(a_{ij})_{1\leq i\leq s,1\leq j\leq r}$ be a matrix with $s\geq r$ and $a_{ij}\in\cO_X(X)$ for all $i$ and $j$.
We consider the closed subscheme  $Z_A$ of $X$ defined by the ideal generated by the $r$-minors of $A$, which we assume is a proper subscheme. We also consider 
the incidence correspondence $W_A\hookrightarrow Y=X\times {\mathbf P}^{r-1}$, defined by $\big\{\sum_{1\leq j\leq r}a_{ij}y_j\mid 1\leq i\leq s\big\}$, where
$y_1,\ldots,y_r$ are the homogeneous coordinates on ${\mathbf P}^{r-1}$.
Our main interest is in the case of a square matrix $A$, when $Z_A$ is a hypersurface in $X$, but our first main result holds in the general setting. 
Our goal is to relate the singularities of the pairs $(X,Z_A)$ and $(Y,W_A)$. 

Recall that if $\Gamma$ is a closed subscheme of a smooth variety $T$, then we may consider the \emph{log canonical threshold} ${\rm lct}(T,\Gamma)$, which
can be defined in terms of a log resolution of the pair $(T,\Gamma)$ (see \cite[Chapter~9.3.B]{Lazarsfeld} for an introduction to log canonical thresholds). This is a fundamental invariant
of singularities of pairs, which can be defined in a more general setting (only assuming that the pair $(T,\Gamma)$ has mild singularities), and which plays an important role
in birational geometry.
It is well-known and easy to see
that if $\Gamma$ is locally defined by $d$ equations, then ${\rm lct}(T,\Gamma)\leq d$. Moreover,
if equality holds, then $\Gamma$ is a local complete intersection of pure codimension $d$ in $T$.
In the setting of interest for us, we prove the following result relating the log canonical thresholds of $(X,Z_A)$ and $(Y,W_A)$:

\begin{thm}\label{thm1_main}
With the above notation, the following hold:
\begin{enumerate}
\item[i)] If ${\rm lct}(X,Z_A)\geq c$, then ${\rm lct}(Y,W_A)\geq \min\{rc,r-1+c\}$.
\item[ii)] If ${\rm lct}(Y,W_A)\geq r-1+c'$, then ${\rm lct}(X,Z_A)\geq \min\big\{c',\tfrac{c'-1}{r}+1\big\}$.
\end{enumerate}
\end{thm}

Note that we always have ${\rm lct}(X,Z_A)\leq {\rm codim}_X(Z_A)\leq s-r+1$, and both inequalities are equalities in the generic case, that is, 
when $X$ is the affine space of $s\times r$ matrices and $A=(x_{ij})_{1\leq i\leq s,1\leq j\leq r}$
(see \cite[Theorem~D]{Docampo}). If ${\rm lct}(X,Z_A)=s-r+1$, then assertion i) in Theorem~\ref{thm1_main} implies that ${\rm lct}(Y,W_A)=s$.
In particular, in this case $W_A$ is a complete intersection in $Y$, of pure codimension $s$. Of course, in the generic case, we know that
$W_A$ is smooth of codimension $s$ in $Y$. This case shows that the inequality ii) in Theorem~\ref{thm1_main} is not optimal when $s>r\geq 2$.
However, for square matrices, we get

\begin{cor}\label{cor}
If $r=s$, then ${\rm lct}(X,Z_A)=1$ if and only ${\rm lct}(Y,W_A)=r$.
\end{cor}

For square matrices, we can go further and relate the property of having rational singularities for $Z_A$ and $W_A$:

\begin{thm}\label{thm2_main}
If $r=s$, then $Z_A$ has rational singularities if and only if $W_A$ is a local complete intersection of pure codimension $r$ with rational singularities.
\end{thm}

For example, in the case of the generic determinantal hypersurface, 
$W_A$ is smooth, of codimension $r$ in $X\times {\mathbf P}^{r-1}$, and we recover Kempf's result \cite{Kempf} that $Z_A$ has rational singularities. 
In fact, it is not hard to see that the ``if" implication in Theorem~\ref{thm2_main} can be deduced via Kempf's approach in \emph{loc. cit}., at least if we know \emph{a priori} that 
$Z_A$ is reduced (see
Remark~\ref{rmk:KempfDirectImage}).

A more interesting example occurs for configuration hypersurfaces, in the sense of \cite{BEK} (we recall the definition in Section~\ref{application} below). In this case, 
the equation $f={\rm det}(A)$ is a square-free polynomial. Moreover, if the matroid corresponding to the hypersurface is connected, then $f$ is irreducible, and therefore
the hypersurface $Z_A$ has rational singularities by \cite[Theorem~1.1]{BMW}, and therefore $W_A$ has rational singularities by Theorem~\ref{thm2_main}. This will be revisited in \cite{BDSW} in arbitrary characteristic, see Remark \ref{rmk:charp}.

The proofs of Theorems~\ref{thm1_main} and \ref{thm2_main} make use of the description of local complete intersection rational singularities and of the log canonical threshold via the codimension 
of certain contact loci in the space of arcs of the ambient variety (see \cite{Mustata1} and \cite{Mustata2}, and also \cite{ELM}). The space of arcs $T_{\infty}$ of a smooth
variety $T$ parametrizes morphisms ${\rm Spec}\,\C\llbracket t\rrbracket\to T$. Given a proper closed subscheme $\Gamma$ of $T$ defined by the ideal $I_{\Gamma}$, the log canonical threshold of $(T,\Gamma)$ can be described 
in terms of the codimensions of the contact loci
$${\rm Cont}^{m}(\Gamma):=\big\{\gamma\in T_{\infty}\mid {\rm ord}_t\gamma^{-1}(I_T)=m\big\}.$$
Similarly, if $\Gamma$ is a local complete intersection in $T$, then one can describe when $\Gamma$ has rational singularities in terms of the contact loci. 
In our setting, given an arc $\gamma\in X_{\infty}\smallsetminus (Z_A)_{\infty}$, we can consider the pull-back $\gamma^*(A)$ of $A$ via $\gamma$, which is an $s\times r$ matrix with 
entries in $\C\llbracket t\rrbracket$, which has a nonzero $r$-minor. 
By the structure theorem for matrices over PIDs, we see that up to the left action of ${\rm GL}_s\big(\C\llbracket t\rrbracket\big)$
and the right action of ${\rm GL}_r\big(\C\llbracket t\rrbracket\big)$, the matrix $\gamma^*(A)$ can be written as a matrix $(c_{ij})$, with $c_{ij}=0$ for $i\neq j$ and
$c_{ii}=t^{\lambda_i}$ for $1\leq i\leq r$,
for
some $\lambda=(\lambda_1,\ldots,\lambda_r)$, with $\lambda_1\leq\ldots\leq\lambda_r$. The key point is to consider the stratification of $X_{\infty}\smallsetminus (Z_A)_{\infty}$
by the loci $C_A(\lambda)$ corresponding to a given $\lambda$. We have $C_A(\lambda)\subseteq {\rm Cont}^{|\lambda|}(Z_A)$, where $|\lambda|=\sum_i\lambda_i$, and
inside each $C_A(\lambda)\times {\mathbf P}^{r-1}_{\infty}\subseteq Y_{\infty}$, one can easily describe 
the codimension of the set of arcs with order $m$ along $W_A$. We note that such a stratification of the space of arcs by the type of the pull-back of a matrix was first considered by Docampo in 
\cite{Docampo} in order to describe the contact loci of generic determinantal varieties, and it was later used by Zhu in \cite{Zhu} in order to give a jet-theoretic proof of Kempf's result \cite{Kempf}
saying that the theta divisor on the Jacobian of a smooth projective curve has rational singularities. 

The paper is organized as follows. In the next section, we recall the connection between singularities of pairs and contact loci in arc spaces, following \cite{ELM}. In
Section~\ref{section_main} we give the proofs of Theorems~\ref{thm1_main} and \ref{thm2_main}, and in Section~\ref{application} we discuss the application to 
configuration hypersurfaces.

\subsection*{\bf Acknowledgment} We are indebted to Uli Walther for many discussions related to the topic of this note. 

\section{Singularities and contact loci}\label{section_review}

Let $X$ be a smooth, irreducible, $n$-dimensional complex algebraic variety. Suppose that $Z$ is a proper closed subscheme of $X$, defined by the ideal sheaf $I_Z$. 
In this section, we recall the connection between the singularities of the pair $(X,Z)$ and the contact loci of $Z$, following \cite{ELM}.

Recall that for every $m\geq 0$, the jet scheme $X_m$ is a variety whose (closed) points parametrize morphisms
${\rm Spec}\,\C[t]/(t^{m+1})\to X$. In particular, we have $X_0=X$. Truncation induces
canonical morphisms $\pi_{m,p}\colon X_m\to X_p$ for all $m>p$, which are affine. Moreover, since $X$ is smooth, of dimension $n$, it is easy to see that $\pi_{m,p}$ is locally
trivial (in the Zariski topology), with fiber ${\mathbf A}^{(m-p)n}$. The projective limit $X_{\infty}=\projlim_mX_m$ is the arc space of $X$ and its ${\mathbf C}$-valued points
 parametrize arcs to $X$, that is,
morphisms ${\rm Spec}\,\C\llbracket t\rrbracket\to X$. Note that we have morphisms $\pi_m\colon X_{\infty}\to X_m$ such that $\pi_p=\pi_{m,p}\circ\pi_m$ if $m>p$. 
When $X$ is not clear from the context, we write $\pi_m^X$ instead of $\pi_m$.

A \emph{cylinder} in $X_{\infty}$ is a subset of the form $C=\pi_m^{-1}(S)$, for some constructible subset $S\subseteq X_m$. We say that this is a \emph{locally closed}, \emph{closed}, or \emph{irreducible}
cylinder if $S$ has the respective property. For a locally closed cylinder $C=\pi_m^{-1}(S)$, we put
$${\rm codim}(C):={\rm codim}_{X_m}(\overline{S}).$$
Of course, all these notions do not depend on the choice of $m$.  It is convenient to make the convention that ${\rm codim}(\emptyset)=\infty$.

Given a closed subscheme $Z$ of $X$, defined by the ideal $I_Z$, we have a function ${\rm ord}_Z\colon X_{\infty}\to {\mathbf Z}_{\geq 0}\cup\{\infty\}$ defined as follows. If $\gamma\in X_{\infty}$, then the inverse image
$\gamma^{-1}(I_Z)$ is an ideal in $\C\llbracket t\rrbracket$, hence it is equal to $(t^m)$ for some $m$ (with the convention that $m=\infty$ if the ideal is $0$). We put ${\rm ord}_Z(\gamma)=m$. 
The contact loci of $Z$ are given by 
$${\rm Cont}^{\geq m}(Z)={\rm ord}_Z^{-1}(\geq m)\quad\text{and}\quad {\rm Cont}^m(Z)={\rm ord}_Z^{-1}(m).$$
Note that ${\rm Cont}^{\geq m}(Z)$ is a closed cylinder and ${\rm Cont}^m(Z)$ is a locally closed cylinder.
On the other hand, if $C$ is any irreducible locally closed cylinder in $X_{\infty}$, we put
$${\rm ord}_C(Z):=\min\big\{{\rm ord}_Z(\gamma)\mid\gamma\in C\big\}\in\Z_{\geq 0}$$
(the fact that this is finite is proved in \cite{ELM}, using the fact that $C$ is a cylinder).
Note that if ${\rm ord}_C(Z)=m$, then we have a subcylinder $C_0\subseteq C$ that is open in $C$ (and thus ${\rm codim}(C_0)={\rm codim}(C)$)
and $C_0\subseteq {\rm Cont}^m(Z)$.

It is shown in \cite{ELM} that if $C$ is an irreducible locally closed cylinder that does not dominate $X$, then ${\rm ord}_C$ 
gives an integer multiple of a divisorial valuation on the function field of $X$. Moreover, every such valuation arises from some cylinder 
as above. The log discrepancy of the valuation is related to the codimension of the cylinder, and this allows one to reformulate certain properties
of the singularities of $(X,Z)$ in terms of codimension of cyclinders.

For the following description of the log canonical threshold ${\rm lct}(X,Z)$, see \cite[Theorem~0.1]{Mustata2} and \cite[Section~2]{ELM}.

\begin{thm}\label{thm_old1}
If $Z$ is a proper closed subscheme in the smooth, irreducible complex algebraic variety $X$, and if $c>0$, then the following are equivalent:
\begin{enumerate}
\item[i)] ${\rm lct}(X,Z)\geq c$.
\item[ii)] ${\rm codim}\big({\rm Cont}^m(Z)\big)\geq cm$ for all $m\geq 0$.
\item[iii)] ${\rm codim}\big({\rm Cont}^{\geq m}(Z)\big)\geq cm$ for all $m\geq 0$.
\item[iv)] ${\rm codim}(C)\geq c\cdot {\rm ord}_C(Z)$ for every locally closed irreducible cylinder $C\subseteq X_{\infty}$. 
\end{enumerate}
\end{thm}

\begin{rmk}\label{cylinders_not_in_arc_space}
For every proper closed subscheme $X'$ of $X$, we have a natural embedding $X'_{\infty}\hookrightarrow X_{\infty}$. Given any such $X'$, in assertion iv) in the above theorem,
it is enough to only consider cylinders $C$ such that $C\cap X'_{\infty}=\emptyset$. In fact, for every locally closed cylinder $C$, we have
$${\rm codim}(C)=\min\big\{{\rm codim}(C\cap {\rm Cont}^m(X'))\mid C\cap {\rm Cont}^m(X')\neq\emptyset\big\}$$
(see \cite[Proposition~1.10]{ELM}). 
\end{rmk}

For the following characterization of rational singularities for local complete intersections, see \cite[Theorem~0.1]{Mustata1} and \cite[Section~2]{ELM}.

\begin{thm}\label{thm_old2}
Let $Z$ be a reduced and irreducible proper closed subscheme in the smooth, irreducible complex algebraic variety $X$. If $Z$ is a local complete intersection, of
pure codimension $r$ in $X$, then
the following are equivalent:
\begin{enumerate}
\item[i)] $Z$ has rational singularities.
\item[ii)] For every locally closed cylinder $C\subseteq X_{\infty}$ that dominates a proper closed subset of $Z$,
we have
$${\rm codim}(C)>r\cdot {\rm ord}_C(Z).$$
\end{enumerate}
\end{thm}

For future reference, we record the following well-known fact:

\begin{prop}\label{max_lct}
If $X$ is a smooth, irreducible complex algebraic variety, and $Z$ is a closed subscheme of $X$ that can be locally defined by $r$ equations, then ${\rm lct}(X,Z)\leq r$,
and if equality holds, then $Z$ is reduced and a local complete intersection of pure codimension $r$. 
\end{prop}

\begin{proof}
If $Z_{\rm red}\hookrightarrow Z$ is the reduced scheme with the same support as $Z$, then ${\rm lct}(X,Z)\leq {\rm lct}(X,Z_{\rm red})$. 
If $Z'$ is any irreducible component of $Z_{\rm red}$ and $U$ is an open subset of $X$ such that $Z'\cap U$ is smooth, then
$${\rm lct}(X,Z_{\rm red})\leq {\rm lct}(U,Z'\cap U)={\rm codim}_X(Z').$$
Since $Z$ is locally cut out by $r$ equations, it follows that ${\rm codim}_X(Z')\leq r$ and thus ${\rm lct}(X,Z)\leq r$.
Moreover, if ${\rm lct}(X,Z)=r$, then we see that every irreducible component of $Z$ has codimension $r$, and thus $Z$ is a local complete intersection. 
In particular, $Z$ is Cohen-Macaulay, hence in order to prove that it is reduced, it is enough to show that it is generically reduced. This is easy
(see for example \cite[Remark~4.2]{CDMO}). 
\end{proof}

\section{Proofs of the main results}\label{section_main}

Let $X$ be a smooth, irreducible complex algebraic variety, with $\dim(X)=n$. Suppose we have a matrix $A=(a_{ij})_{1\leq i\leq s,1\leq j\leq r}$, with $s\geq r$, and
$a_{ij}\in\cO_X(X)$ for all $i$ and $j$, such that the ideal generated by the $r$-minors of $A$ is nonzero, and thus defines a proper closed subscheme $Z_A$ of $X$.

We begin by setting up some notation. For every arc $\gamma\colon {\rm Spec}\,\C\llbracket t\rrbracket\to X$, we consider the pull-back
$\gamma^*(A)$, which is an $s\times r$ matrix with entries in $\C\llbracket t\rrbracket$. For simplicity, let's assume that $\gamma\not\in (Z_A)_{\infty}$,
so some $r$-minor of $\gamma^*(A)$ is nonzero.
By the structure theorem for matrices over PIDs, we know that there are an $r\times r$-matrix $Q$ and an $s\times s$-matrix $P$, with entries in $\C\llbracket t\rrbracket$,
both invertible, 
such that
\begin{equation}\label{structure_matrix}
P\cdot\gamma^*(A)\cdot Q={\rm diag}(t^{\lambda_1},\ldots,t^{\lambda_r}),
\end{equation}
where we denote by ${\rm diag}(c_1,\ldots,c_r)$ the $s\times r$-matrix $(c_{ij})$ with $c_{ii}=c_i$ for $1\leq i\leq r$, and $c_{i,j}=0$ for $i\neq j$.
In (\ref{structure_matrix}), we may and will assume that
$0\leq\lambda_1\leq\ldots\leq\lambda_r<\infty$.
 Given such $\lambda=(\lambda_1,\ldots,\lambda_r)$,
we denote by $C_A(\lambda)$ the set of arcs $\gamma$ such that $\gamma^*(A)$ satisfies (\ref{structure_matrix}) for suitable $P$ and $Q$.
Note that $C_A(\lambda)$ is a locally closed cylinder in $X_{\infty}$: if we denote by $Z_{\ell}$ the subscheme of $X$ defined by the ideal generated by the $\ell$-minors of $A$,
then
$$C_A(\lambda)=\bigcap_{\ell=1}^r{\rm Cont}^{\lambda_1+\ldots+\lambda_{\ell}}(Z_{\ell}).$$
In particular, we see that $C_A(\lambda)\subseteq {\rm Cont}^{|\lambda|}(Z_A)$, where we put $|\lambda|=\lambda_1+\ldots+\lambda_r$.
In fact, for every $m$ we have
\begin{equation}\label{eq_disjoint_decomposition}
{\rm Cont}^m(Z_A)=\bigsqcup_{|\lambda|=m}C_A(\lambda),
\end{equation}
and thus $\codim \Cont^m(Z_A) = \min_{|\lambda| = m} \codim C_A(\lambda)$. 

Recall now that given a matrix $A$ as above, we consider the closed subscheme $W_A$ of $Y=X\times {\mathbf P}^{r-1}$ 
defined by the ideal $J$ generated by $\sum_{j=1}^ra_{ij}y_j$, for $1\leq i\leq s$, where we denote by $y_1,\ldots,y_r$ the homogeneous coordinates on ${\mathbf P}^{r-1}$.

We can now give the proofs of the theorems stated in the Introduction.

\begin{proof}[Proof of Theorem~\ref{thm1_main}]
It follows from the previous discussion and Theorem~\ref{thm_old1} that
\begin{equation}\label{eq_char_lct_Z}
{\rm lct}(X,Z_A)\geq c\quad\text{if and only if}\quad {\rm codim}\big(C_A(\lambda)\big)\geq c|\lambda|\,\,\text{for all}\,\,\lambda.
\end{equation}
We next need to translate the condition ${\rm lct}(Y,W_A)\geq c'$.

Note that we have a canonical isomorphism $Y_{\infty}\simeq X_{\infty}\times {\mathbf P}^{r-1}_{\infty}$. 
Since we have the decomposition 
$${\rm Cont}^m(W_A)\smallsetminus \big((Z_A)_{\infty}\times {\mathbf P}^{r-1}_{\infty}\big)=\bigsqcup_{\lambda}{\rm Cont}^m(W_A)_{\lambda},$$
where
$${\rm Cont}^m(W_A)_{\lambda}=
 {\rm Cont}^m(W_A)\cap \big(C_A(\lambda)\times {\mathbf P}^{r-1}_{\infty}\big),$$
 it follows from \cite[Proposition~1.10]{ELM} that
 $${\rm codim}\big({\rm Cont}^m(W_A)\big)=\min_{\lambda}{\rm codim}\big({\rm Cont}^m(W_A)_{\lambda}\big).$$
We conclude using Theorem~\ref{thm_old1} that
 \begin{equation}\label{eq_char_lct_W1}
 {\rm lct}(Y,W_A)\geq c'\quad\text{if and only if}\quad {\rm codim}\big({\rm Cont}^m(W_A)_{\lambda}\big)\geq c'm\,\,\text{for all}\,\,m\,\,\text{and all}\,\,\lambda.
 \end{equation}
 
 The key point is to compute, for a given $\gamma\in C_A(\lambda)$, the codimension of 
 $${\rm Cont}^{\geq m}(W_A)_{\gamma}:=(\{\gamma\}\times {\mathbf P}^{r-1}_{\infty})\cap {\rm Cont}^{\geq m}(W_A),$$
 that we view as a locally closed cylinder in ${\mathbf P}^{r-1}_{\infty}$. 
 Note that we can think of ${\mathbf P}^{r-1}_{\infty}$ in two ways. On the one hand, the affine open cover ${\mathbf P}^{r-1}=U_1\cup\ldots\cup U_r$, where
 $U_i=(y_i\neq 0)$, induces a corresponding affine open cover ${\mathbf P}^{r-1}_{\infty}=(U_1)_{\infty}\cup\ldots\cup (U_r)_{\infty}$. 
 On the other hand, we can think of ${\mathbf P}^{r-1}_{\infty}$ as the quotient of the set of elements $u=(u_1,\ldots,u_n)\in\big(\C\llbracket t\rrbracket\big)^r$,
 with some $u_i$ being invertible, modulo the action of invertible elements in $\C\llbracket t\rrbracket$. 
 Via this latter description, the subset $(U_i)_{\infty}$ corresponds to those $u$ such that $u_i$ is invertible, in which case we may assume that it is $1$. We also note that, via this description,
 ${\rm Cont}^{\geq m}(W_A)$ consists of those pairs $(\gamma,u)$ such that if $\gamma^*(A)=(w_{ij})\in M_{s,r}\big(\C\llbracket t\rrbracket\big)$, then 
 ${\rm ord}\big(\sum_jw_{ij}u_j\big)\geq m$ for $1\leq i\leq s$. 
 
 Given any $\gamma\in C_A(\lambda)$, we can find invertible
 matrices $P\in M_s\big(\C\llbracket t\rrbracket\big)$ and $Q\in M_r\big(\C\llbracket t\rrbracket\big)$ such that $P\cdot\gamma^*(A)\cdot Q={\rm diag}(t^{\lambda_1},\ldots,t^{\lambda_r})$. 
It is straightforward to check that if $P=(c_{ij})$ and 
$u=(u_1,\ldots,u_r)\in \big(\C\llbracket t\rrbracket\big)^r$, then 
$$\min_i{\rm ord}(u_i)=\min_i{\rm ord}\big(\sum_jc_{ij}u_j\big).$$
Moreover, if $C$ is a locally closed cylinder in ${\mathbf P}^{r-1}_{\infty}$ and $Q\cdot C$ is the image of $C$ via the automorphism of ${\mathbf P}^{r-1}_{\infty}$ induced by $Q$,
then ${\rm codim}(Q\cdot C)={\rm codim}(C)$. 

Therefore, in order to compute ${\rm codim}\big({\rm Cont}^{\geq m}(W_A)_{\gamma}\big)$, we may and will assume that
$\gamma^*(A)={\rm diag}(t^{\lambda_1},\ldots,t^{\lambda_r})$. 
Let us consider ${\rm Cont}^{\geq m}(W_A)_{\gamma}\cap (U_i)_{\infty}$. This consists of those
$u=(u_1,\ldots,1,\ldots,u_r)\in \big(\C\llbracket t\rrbracket\big)^{r-1}$ 
(with $u_j=1$ for $j=i$) such that 
${\rm ord}(t^{\lambda_j}u_j)\geq m$ for all $j$. This is equivalent to $\lambda_i\geq m$ and ${\rm ord}(u_k)\geq m-\lambda_k$ for $1\leq k\leq i-1$.
We see that this is nonempty if and only if $\lambda_i\geq m$, and in this case ${\rm Cont}^{\geq m}(W_A)_{\gamma}\cap (U_i)_{\infty}$ is
an irreducible closed cylinder in $(U_i)_{\infty}$, of codimension
$\sum_{j;\lambda_j<m}(m-\lambda_j)$. By letting $i$ vary, we conclude that ${\rm Cont}^{\geq m}(W_A)_{\gamma}$ is nonempty if and only if $\lambda_r\geq m$,
and in this case it is an irreducible closed cylinder in ${\mathbf P}^{r-1}_{\infty}$, with
$${\rm codim}\big({\rm Cont}^{\geq m}(W_A)_{\gamma}\big)=\sum_{j;\lambda_j<m}(m-\lambda_j).$$
Since this is independent of $\gamma\in C_A(\lambda)$, it follows that 
$$\big(C_A(\lambda)\times {\mathbf P}^{r-1}_{\infty}\big)\cap {\rm Cont}^{\geq m}(W_A)\neq \emptyset\quad\text{if and only if}\quad \lambda_r\geq m,$$
and if this is the case, then its codimension in $Y_{\infty}$ is equal to 
$${\rm codim}\big(C_A(\lambda)\big)+\sum_{j;\lambda_j<m}(m-\lambda_j).$$
We deduce using Theorem~\ref{thm_old1} that
${\rm lct}(Y,W_A)\geq c'$ if and only if for all $\lambda$ and all $m\leq\lambda_r$, we have
\begin{equation}\label{eq_lct_W}
{\rm codim}\big(C_A(\lambda)\big)+\sum_{j;\lambda_j<m}(m-\lambda_j)\geq c'm.
\end{equation}

Suppose first that ${\rm lct}(X,Z_A)\geq c$, hence ${\rm codim}\big(C_A(\lambda)\big)\geq c(\lambda_1+\ldots+\lambda_r)$ for all $\lambda$, with $\lambda_1\leq\ldots\leq \lambda_r$
by (\ref{eq_char_lct_Z}). If $m\leq \lambda_r$, then
\begin{equation}\label{ineq100}
{\rm codim}\big(C_A(\lambda)\big)+\sum_{j;\lambda_j<m}(m-\lambda_j) \geq c(\lambda_1+\ldots+\lambda_r)+\sum_{j;\lambda_j<m}(m-\lambda_j).
\end{equation}
In order to prove i), it is enough to show that the right-hand side of (\ref{ineq100}) is bounded below by $m\cdot\min\{rc,r-1+c\}$. 
Note that $\min\{rc,r-1+c\}=rc$ if $c\leq 1$ and $\min\{rc,r-1+c\}=r-1+c$ if $c\geq 1$.

Let $p=\max\{i\mid \lambda_i<m\}$ (note that if $\lambda_i\geq m$ for all $i$, then the right-hand side of (\ref{ineq100}) is clearly $\geq rcm$). With this notation, the right-hand side of (\ref{ineq100}) is equal to
$$
c(\lambda_1+\ldots+\lambda_p)+c(\lambda_{p+1}+\ldots+\lambda_r)+(m-\lambda_1)+\ldots+(m-\lambda_p)
$$
\begin{equation}\label{ineq111}
\geq c(r-p)m+pm-(1-c)\sum_{i=1}^p\lambda_i.
\end{equation}
Note that if $c\leq 1$, then the expression in (\ref{ineq111}) is 
$$rcm+(1-c)\left(pm-\sum_{i=1}^p\lambda_i\right)\geq rcm,$$ since $\lambda_i\leq m$ for $1\leq i\leq p$.
On the other hand, if $c\geq 1$, then using the fact that $p\leq r-1$, we see that the expression in (\ref{ineq111}) is
$$\geq c(r-p)m+pm=\big(cr+p(1-c)\big)m\geq\big(cr+(r-1)(1-c)\big)m=(r-1+c)m.$$
This completes the proof of i). 

Suppose now that ${\rm lct}(Y,W_A)\geq r-1+c'$. We need to show that ${\rm lct}(X,Z_A)\geq c'$ if $c'\leq 1$ and ${\rm lct}(X,Z_A)\geq 
\tfrac{c'-1}{r}+1$ if $c'\geq 1$. 
Given $\lambda$, with $\lambda_1\leq\ldots\leq\lambda_r$, let $m=\lambda_r$, so (\ref{eq_lct_W}) gives
$${\rm codim}\big(C_A(\lambda)\big)+\sum_{i=1}^r(\lambda_r-\lambda_i)\geq (r-1+c')\lambda_r,$$
and thus
\begin{equation}\label{ineq112}
{\rm codim}\big(C_A(\lambda)\big)\geq \lambda_1+\ldots+\lambda_{r-1}+c'\lambda_r=|\lambda|+(c'-1)\lambda_r.
\end{equation}
Therefore, if $c'\leq 1$, since $|\lambda|\geq\lambda_r$, we have
$${\rm codim}\big(C_A(\lambda)\big)\geq c'|\lambda|.$$
On the other hand, if $c'\geq 1$, using the fact that $\lambda_r\geq\tfrac{1}{r}|\lambda|$, we get
$${\rm codim}\big(C_A(\lambda)\big)\geq \big(\tfrac{c'-1}{r}+1\big)|\lambda|.$$
We now deduce the conclusion in ii) using (\ref{eq_char_lct_Z}).
\end{proof}

The assertion in Corollary~\ref{cor} follows directly from that in Theorem~\ref{thm1_main} since when $r=s$,
we always have ${\rm lct}(X,Z_A)\leq 1$ and ${\rm lct}(Y,W_A)\leq r$ by Proposition~\ref{max_lct}.
From now on we assume that we have a square matrix and relate
the conditions for $Z_A$ and $W_A$ to have rational singularities.
Recall that for every scheme $T$, we denote by $\pi_0^T$ the canonical morphism $T_{\infty}\to T$.

\begin{proof}[Proof of Theorem~\ref{thm2_main}]
Suppose first that $Z_A$ has rational singularities. In particular, $Z_A$ is reduced and its irreducible components do not intersect. After replacing $X$ by 
the open subsets in a suitable open cover, we may and will assume that $Z_A$ is also irreducible. Using Theorem~\ref{thm_old2}, we see that if
$(Z_A)_{\rm sing}$ is the singular locus of $Z_A$, then 
\begin{equation}\label{eq_cond_rat1}
{\rm codim}\big(C_A(\lambda)\cap(\pi^X_0)^{-1}((Z_A)_{\rm sing})\big)>|\lambda|\quad\text{for all}\quad\lambda.
\end{equation}

On the other hand, since $Z_A$ has rational singularities, it follows that ${\rm lct}(X,Z_A)=1$. We deduce from Theorem~\ref{thm1_main} that ${\rm lct}(Y,W_A)=r$. Since $W_A$
is locally defined in $Y$ by $r$ equations, it follows from Proposition~\ref{max_lct} that $W_A$ is reduced and a local complete intersection in $Y$, of pure codimension $r$. 

We note that 
$$(Z_A)_{\rm sm}:=Z_A\smallsetminus (Z_A)_{\rm sing}\subseteq\big\{x\in X\mid {\rm rank}\,A(x)=r-1\big\}.$$
Indeed, if ${\rm rank}\,A(x)=r'$, then we may assume that $A(x)={\rm diag}(1,\ldots,1,0,\ldots,0)$, with $r'$ nonzero entries, and it is easy to see that
in this case ${\rm mult}_x(Z_A)\geq r-r'$, hence $x$ being a smooth point of $Z_A$ implies $r'=r-1$.
This implies that the morphism $p\colon W_A\to Z_A$ induced by the first projection is an isomorphism over $(Z_A)_{\rm sm}$. Since $(Z_A)_{\rm sm}$
is an open subset of $Z_A$, which is irreducible, we conclude that if $W_A$ is reducible, then some irreducible component of
$W_A$ lies over $(Z_A)_{\rm sing}$.

Note that for every locally closed cylinder $C\subseteq {\rm Cont}^{\geq m}(W_A)$, whose image in $X$ lies in $(Z_A)_{\rm sing}$,
we have
\begin{equation}\label{eq_cond_rat2}
{\rm codim}(C)>rm.
\end{equation}
Indeed, after taking a suitable decomposition of $C$, we may assume that the image of $C$ in $X_{\infty}$ lies in some $C_A(\lambda)$
(see Remark~\ref{cylinders_not_in_arc_space}), and we then argue 
as in the proof of Theorem~\ref{thm1_main}, using the fact that we have the strict inequality in (\ref{eq_cond_rat1}). 
First, (\ref{eq_cond_rat2}) implies that $W_A$ is irreducible: otherwise, we have seen that we have an irreducible component $T$ of $W_A$ (whose codimension in $Y$ is $r$)
that lies over $(Z_A)_{\rm sing}$; if we take $C={\rm Cont}^{\geq m}(W_A)\cap(\pi_0^Y)^{-1}(U)$ for some $m\geq 1$ and some smooth open subset $U$ of $W_A$ contained
in $T$, then 
${\rm codim}(C)=rm$, contradicting (\ref{eq_cond_rat2}). 
Second, we see that for every locally closed cylinder $C\subseteq {\rm Cont}^{\geq m}(W_A)$ 
whose image in $Y$ lies in a proper closed subset of $W_A$,
the 
inequality (\ref{eq_cond_rat2}) holds: indeed, if the image of $C$ in $X$ is not contained in $(Z_A)_{\rm sing}$, then after replacing $X$ by a suitable open subset,
we may assume that both $Z_A$ and $W_A$ are smooth, in which case the assertion is clear. We can now apply Theorem~\ref{thm_old2} to conclude that
$W_A$ has rational singularities.

Conversely, let us assume that $W_A$ is a local complete intersection of pure codimension $r$ in $Y$, with rational singularities. In particular,
we know that $W_A$ is reduced and ${\rm lct}(Y,W_A)=r$, hence ${\rm lct}(X,Z_A)=1$ by Theorem~\ref{thm1_main}. Therefore, $Z_A$ is reduced. 

As before, we consider the morphism $p\colon W_A\to Z_A$ induced by the first projection. 
We note that if $W_1,\ldots,W_d$ are the irreducible components of $W_A$, since $W_A$ has rational singularities, we have $W_i\cap W_j=\emptyset$ for all $i\neq j$. Since $p^{-1}(x)$ is irreducible
for every $x\in Z_A$ (since it is a linear subspace of ${\mathbf P}^{r-1}$), it follows that $p(W_i)\cap p(W_j)=\emptyset$. Since $p$ is a proper morphism, each $p(W_i)$ is closed in $X$,
hence after replacing $X$ by the open subsets in a suitable open cover, we may and will assume that $W_A$ is irreducible, hence $Z_A=p(W_A)$ is irreducible too.

By Theorem~\ref{thm_old2}, in order to show that $Z_A$ has rational singularities, it is enough to show that if $C\subseteq (\pi_0^X)^{-1}(Z_A)$ is an irreducible locally closed cylinder
whose image in $X$ is contained in a proper closed subset $Z'$ of $Z_A$, 
and $m={\rm ord}_C(Z_A)$, then ${\rm codim}(C)>m$. 
By Remark~\ref{cylinders_not_in_arc_space}, we may assume that $C\subseteq C_A(\lambda)$, for some $\lambda=(\lambda_1,\ldots,\lambda_r)$,
with $\lambda_1\leq\ldots\leq\lambda_r$. Let $C'={\rm Cont}^{\geq\lambda_r}(W_A)\cap (C\times{\mathbf P}^{r-1}_{\infty})$. 
The computation in the proof of Theorem~\ref{thm1_main} implies that
$${\rm codim}(C')={\rm codim}(C)+\sum_{i=1}^r(\lambda_r-\lambda_i).$$
Moreover, since $C\subseteq (\pi^X_0)^{-1}(Z')$, it follows that
$C'\subseteq (\pi^Y_0)^{-1}(W')$, where $W'=p^{-1}(Z')$ is a proper closed subset of $W_A$. Since $W_A$ has rational singularities,
it follows from Theorem~\ref{thm_old2} that ${\rm codim}(C')>r\lambda_r$, hence
$${\rm codim}(C)={\rm codim}(C')-\sum_{i=1}^r(\lambda_r-\lambda_i)>r\lambda_r-\sum_{i=1}^r(\lambda_r-\lambda_i)=|\lambda|=m,$$
which completes the proof of the theorem.
\end{proof}

\begin{rmk} \label{rmk:KempfDirectImage}
One can give another proof of the ``if'' implication in Theorem \ref{thm2_main}, at least when we know \emph{a priori} that $Z_A$ is reduced,
using Kempf's argument for \cite[Proposition~2]{Kempf}, as follows. 
Recall that the projection $p\colon Y = X \times {\mathbf P}^{r-1} \to X$ onto the first factor induces $W_A \to Z_A$. On $Y$ we have the line bundle $\cO_Y(1)$,
the pull-back of $\cO_{{\mathbf P}^{r-1}}(1)$ via the second projection.
    
 Suppose first only that $W_A$ is a local complete intersection of pure codimension $r$ in $Y$. Consider the $r$ global sections $q_i = \sum_{1 \leq j \leq r} a_{ij} y_j$
 of  $\cO_Y(1)$, for $1 \leq t \leq r$. These give $r$ morphisms $\cO_Y(-1) \to \cO_Y$ and a corresponding Koszul complex 
 $$0 \to K_r=\cO_Y(-r)\overset{\varphi_r}\to \ldots \to K_1=\cO_Y(-1)^{\oplus r}\overset{\varphi_1}\to K_0=\cO_Y\to 0.$$ 
 The assumption on $W_A$ implies that this is an acyclic complex, giving a resolution of
 $\cO_{W_A}$. By breaking the Koszul complex into short exact sequences, taking the corresponding long exact sequences for higher direct images, and using the standard
 computation of cohomology of the projective space, it follows by descending induction on $j$ that
 $R^qp_*{\rm Im}(\varphi_j)=0$ for $1\leq j\leq r$ and $q>j-1$. Using this for $j=1$, we deduce that 
 $R^qp_*\cO_{W_A}=0$ for all $q>0$ and the induced morphism
 $\cO_X\to p_*\cO_{W_A}$ is surjective. 
 
 Suppose now that $W_A$ is a local complete intersection of pure codimension $r$ in $Y$, with rational singularities. 
 This implies that $W_A$ is reduced and, as in the proof of Theorem \ref{thm2_main}, we see that we may assume that $W_A$ is also irreducible.
 This implies that $Z_A$ is irreducible, too. If we assume that $Z_A$ is reduced (for example, this follows using Theorem~\ref{thm1_main}), then we see 
 that $W_A\to Z_A$ is a proper birational morphism of algebraic varieties. What we have seen so far about $R^qp_*\cO_{W_A}$ says that the canonical morphism 
 $\cO_{Z_A}\to {\mathbf R}p_*\cO_{W_A}$ is an isomorphism.
 
  If $\tau\colon V\to W_A$ is a resolution of singularities, then the canonical morphism
 $\cO_{W_A}\to {\mathbf R}\tau_*\cO_V$ is an isomorphism by the assumption on the singularities of $W_A$, hence 
 the morphism
 $$\cO_{Z_A}\to {\mathbf R}(p\circ\tau)_*\cO_V\simeq {\mathbf R}p_*({\mathbf R}\tau_*\cO_V)\simeq {\mathbf R}p_*\cO_{W_A}$$
 is an isomorphism. Since $p\circ \tau$ is a resolution of singularities of $Z_A$, we conclude that $Z_A$ has rational singularities. 
 \end{rmk}

\begin{rmk}\label{rmk_affine_version}
Given the setup in Theorems~\ref{thm1_main} and \ref{thm2_main} (with $s=r$, for simplicity), we may also consider the subscheme $\widetilde{W}_A$ of $\widetilde{Y}=X\times {\mathbf A}^r$ defined by the same equations 
as $W_A$. Since $X\times\{0\}$ has codimension $r$ in $\widetilde{Y}$, we see that $W_A$ is a local complete intersection of pure codimension $r$ in $Y$ if and only if
$\widetilde{W}_A$ is a local complete intersection of pure codimension $r$ in $\widetilde{Y}$. Similarly, we have 
${\rm lct}(\widetilde{Y},\widetilde{W}_A)=r$ if and only if ${\rm lct}(Y,W_A)=r$. 
Indeed, note first that ${\rm lct}(Y,W_A)=r$ if and only if ${\rm lct}(Y',W_A')=r$, where $Y'=X\times \big({\mathbf A}^r\smallsetminus\{0\}\big)$ and $W_A'=\widetilde{W}_A\cap Y'$. 
The ``only if" part of the assertion is then
clear. For the ``if" part,
it is convenient to use the description of the log canonical threshold
in Theorem~\ref{thm_old1}. Note that if $\gamma=(\delta,u)\in \widetilde{Y}_{\infty}\simeq X_{\infty}\times {\mathbf A}^r_{\infty}$ and $u=(u_1,\ldots,u_r)\in {\rm Cont}^p\big(\{0\}\big)$,
then we can write $u_i=t^pv_i$ for $1\leq i\leq r$, such that $v=(v_1,\ldots,v_r)\in Y'_{\infty}$. 
Since the equations defining $\widetilde{W}$ are homogeneous of degree $1$ in the coordinates on ${\mathbf A}^r$, 
we see that $(\delta,u)\in {\rm Cont}^m(\widetilde{W}_A)$ if and only if $(\delta,v)\in {\rm Cont}^{m-p}(W_A')$. This implies that
for every $m$ and $p$, with $p\leq m$, we have
$${\rm codim}\big({\rm Cont}^m(\widetilde{W}_A)\cap {\rm Cont}^p(X\times \{0\})\big)=pr+{\rm codim}\big({\rm Cont}^{m-p}(W_A')\big)\geq pr+r(m-p)=mr,$$
where the inequality follows from the fact that ${\rm lct}(Y',W_A')=r$, by Theorem~\ref{thm_old1}. Another application of the same theorem gives
${\rm lct}(\widetilde{Y},\widetilde{W}_A)\geq r$, and the opposite inequality follows from Proposition~\ref{max_lct}.

We note, however, that $\widetilde{W}_A$ doesn't have rational singularities, unless $Z_A$ is empty (that is, ${\rm det}(A)$ is invertible).  Indeed, $X\times\{0\}$ is an irreducible component of 
$\widetilde{W}_A$ that intersects every irreducible component that dominates an irreducible component of $W_A$. However, the irreducible components of a variety with rational singularities
do not intersect. 
\end{rmk}

\begin{eg}
  Consider a square matrix $A$ whose nonzero entries are homogeneous linear polynomials in ${\mathbf C}[x_1, \dots, x_n]$. In light of Theorem \ref{thm2_main}, it is interesting to characterize those $A$ for which $W_A$ is smooth. It turns out that $W_A$ is smooth if and only if $A$ is $1$-generic. Theorem \ref{thm2_main} thus recovers the fact that hypersurfaces defined by determinants of square $1$-generic matrices have rational singularities. This result is due to Kempf (see the discussion around \cite[Theorem~1]{Eisenbud}). It also recovers Kempf's result that theta divisors of smooth projective curves have rational singularities: by \cite[Theorem~2]{Kempf}, all tangent cones of theta divisors on Jacobians are defined by such determinants.

    Consider a map $\psi\colon U \otimes_{{\mathbf C}} U \to V$, where $U$ and $V$ are ${\mathbf C}$-vector spaces of dimension $r$ and $n$, respectively. After fixing bases 
    $x_1, \dots, x_n$ for $V$ and $b_1, \dots, b_r$ for $U$, this map is represented by an $r \times r$ matrix $A_{\psi} = (a_{ij})$, where each 
    $a_{ij} = \psi(b_i \otimes b_j)$ is a homogeneous linear polynomial in ${\rm Sym}(V)={\mathbf C}[x_1, \dots, x_n]$ (possibly zero). Clearly, giving a square matrix with homogeneous linear entries is 
    equivalent to giving a map $\psi$.
  
Recall that the map $\psi$ is $1$-generic if $\ker \psi$ contains no elementary tensors $b \otimes b^\prime$; we say $A_{\psi}$ is $1$-generic when $\psi$ is $1$-generic. On the matrix side, being $1$-generic is the same as saying $v^{\intercal} A_{\psi} w \neq 0$ for any nonzero column vectors $v,w \in {\mathbf C}^n \smallsetminus \{0\}$. In other words, a matrix is $1$-generic if performing nonzero row and column operations never produces a vanishing entry.
  
We now show that $A_{\psi}$ is $1$-generic if and only if $W_{A_{\psi}} \subseteq Y={\mathbf A}^n \times {\mathbf P}^{r-1}$ is smooth, of codimension $r$ in $Y$. Using homogeneity, we see that $a_{ij} = \sum_{1 \leq k \leq n} \frac{\partial a_{ij}}{x_k} x_k$. Therefore, we have
  \begin{align*} \label{eqn:1-generic-1}
      v^{\intercal} A_{\psi} w \neq 0 
      &\iff v^{\intercal} 
        \begin{bmatrix}
          \sum_{1 \leq j \leq r} \bigg( \sum_{1 \leq k \leq n} \frac{\partial a_{1j}}{\partial x_k} x_k \bigg) w_j \\
          \vdots \\
          \sum_{1 \leq j \leq r} \bigg( \sum_{1 \leq k \leq n} \frac{\partial a_{rj}}{\partial x_k} x_k\bigg) w_j
          \end{bmatrix} \neq 0 \\
        &\iff v^{\intercal}
        \begin{bmatrix}
          \sum_{1 \leq j \leq r} \frac{\partial a_{1j}}{\partial x_1} w_j & \cdots & \sum_{1 \leq j \leq r} \frac{\partial a_{1j}}{\partial x_n} w_j \\
          \vdots & \cdots & \vdots \\
          \sum_{1 \leq j \leq r} \frac{\partial a_{rj}}{\partial x_1} w_j & \cdots & \sum_{1 \leq j \leq r} \frac{\partial a_{rj}}{\partial x_n} w_j
          \end{bmatrix} \neq 0. \nonumber
    \end{align*} 
    The affine cone $\widetilde{W}_{A_{\psi}}\subseteq {\mathbf A}^n \times {\mathbf A}^r$ is cut out by the $r$ quadrics $\sum_{1 \leq j \leq r} a_{tj} y_j$, indexed by $1 \leq t \leq r$. The associated Jacobian matrix is
    \begin{equation*} \label{eqn:1-generic-2}
        \left[ \begin{array}{ccc|ccc}
        a_{11} & \cdots & a_{1r} &\sum_{1 \leq j \leq r} \frac{\partial a_{1j}}{\partial x_1} y_j & \cdots & \sum_{1 \leq j \leq r} \frac{\partial a_{1j}}{\partial x_n} y_j \\
        \vdots & & \vdots & \vdots & & \vdots \\
          a_{r1} & \cdots & a_{rr} & \sum_{1 \leq j \leq r} \frac{\partial a_{nj}}{\partial x_1} y_j & \cdots & \sum_{1 \leq j \leq r} \frac{\partial a_{nj}}{\partial x_n} y_j.
        \end{array} \right]
    = \left[ \begin{array}{c|c} A_{\psi} & B(y)
        \end{array} \right].
    \end{equation*}
    The first (resp. second) block matrix is $A_{\psi}$ (resp. $B(y)$) and comes from taking partial derivatives of the quadrics with respect to the $y_j$ (resp. $x_i$) variables. 
    We have seen that $A_{\psi}$ is 1-generic if and only if $B(w)$ has rank $r$ for every $w\in {\mathbf C}^r\smallsetminus\{0\}$. It follows that if 
    $A_{\psi}$ is $1$-generic, then the singular locus of $\widetilde{W}_{A_{\psi}}$ is contained in ${\mathbf A}^n\times\{0\}$, hence $W_{A_{\psi}}$ is smooth. 
    
    Conversely, if $W_{A_{\psi}}$ is smooth, then every point of $\{0\}\times \big({\mathbf A}^r\smallsetminus\{0\}\big)$ is a smooth point of $\widetilde{W}_{A_{\psi}}$.
    Since the matrix $A_{\psi}$ vanishes at every such point, it follows that the ideal generated by 
    the $r \times r$ minors of $B$ only vanishes at the origin, hence $A_{\psi}$ is $1$-generic. 
    
    Finally, we note that the definition of $1$-generic matrices and its characterization in terms of smoothness of the appropriate incidence variety easily generalize to non-square matrices.
\end{eg}

\begin{rmk}
    Any hypersurface can be defined by a determinant in a trivial way: simply consider a $1 \times 1$ determinant. When $X = {\mathbf A}^n$, every hypersurface can be defined by a determinant in a much less trivial way: it is classical (see, for instance, \cite{HMV}), that given $f \in {\mathbf C}[x_1, \dots, x_n]$, there exists a square matrix $A = (a_{ij})$ such that
 $f={\rm det}(A)$ and every $a_{ij}$ is a polynomial in ${\mathbf C}[x_1, \dots, x_n]$ of degree at most $1$. Note that $A$ is usually quite large and often contains constants and inhomogeneous entries, even if $f$ itself is homogeneous. Therefore, Theorems~\ref{thm1_main} and \ref{thm2_main}
apply to all hypersurfaces in ${\mathbf A}^n$ in a nontrivial way, modulo constructing the matrix $A$.
\end{rmk}

\section{Configuration hypersurfaces}\label{application}

We begin by recalling, following the original \cite{BEK} and the more recent \cite{DSW}, the definition of configuration hypersurfaces. A \emph{configuration} is a nonzero linear subspace $U$ in a finite-dimensional $\C$-vector space 
$V=\C^E$, with a fixed basis $E$. We denote by $V^*$ the dual vector space and by $E^*=\{e^*\mid e\in E\}$ the dual basis. For $S\subseteq E$, we put
$S^*=\{e^*\mid e\in S\}$. 

Given a configuration $U$ as above, the corresponding \emph{configuration matroid} $M_U$ is the matroid on $E$, whose independent sets are those 
$S\subseteq E$ such that the restrictions of the elements of $S^*$ to $U$ are linearly independent. For basic facts about matroids, we refer to \cite{matroids};
however, we will not need any of this in what follows. 

The \emph{Hadamard product} on $V$ is given by $Q_E\colon V\times V\to V$,
$$Q_E\left(\sum_{e\in E}a_ee,\sum_{e\in E}b_ee\right)=\sum_{e\in E}a_eb_ee.$$
Given a configuration $U\subseteq V$, we denote by $Q_U$ the restriction of $Q_E$ to $U\times U$.
Set-theoretically, the \emph{configuration hypersurface} $Z_U\subseteq V^*$ consists of those
$\beta\in V^*$ such that $\beta\circ Q_U$ is degenerate. 

In order to give an explicit equation of $Z_U$, let us order the elements of $E$ as $e_1,\ldots,e_n$, so ${\rm Sym}(V)\simeq \C[x_1,\ldots,x_n]$.
We choose a basis $B$ of $U$, so that with respect to $B$ and $E$, the subspace $U$
is described as the span of the row vectors of an $r\times n$ matrix $D$, where $r=\dim(U)$. Therefore, an equation
of $Z_U$ is given by 
$${\rm det}(A),\quad\text{where}\quad A=D\cdot {\rm diag}(x_1,\ldots,x_n)\cdot D^T.$$
The matrix $A$ is the \emph{Patterson} matrix of the configuration (with respect to the basis $B$). 
As picking a different basis of $U$ changes the determinant of the Patterson matrix by a nonzero scalar, this defining equation depends (up to scalar) only on the configuration 
$U \subseteq \mathbf{C}^E$.
An application of the Cauchy-Binet formula
gives
\begin{equation}\label{eq_Patterson}
{\rm det}(A)=\sum_{I\subseteq E; \#I=r}{\rm det}(D\vert_I)x^I,
\end{equation}
where we denote by $D\vert_I$ the submatrix of $D$ on the columns indexed by $I$ and $x^I=\prod_{i\in I}x_i$.
Note that ${\rm det}(D\vert_I)\neq 0$ if and only if $I$ is a basis of $M_U$, hence ${\rm det}(A)$ is a matroid support polynomial of $M_U$ is the sense of 
\cite[Definition~2.11]{BW}. 

Let ${\mathbf P}(U)$ denote the projective space of lines in $U$.
Set-theoretically, the \emph{configuration incidence variety} $W_U$ is the subset of $V^*\times {\mathbf P}(U)$ consisting of those 
$\big(\beta,[u]\big)$ such that $\beta\circ Q(u,-)$ vanishes on $U$. This incidence variety was first studied by Bloch \cite{Bloch}. In order to give explicit equations for $W_U$, we choose a basis $B$
of $U$ as above. In this case, the ideal defining the subscheme $W_U$ in $V^*\times {\mathbf P}(U)$ is generated by the entries of 
$A\cdot (y_1,\ldots,y_r)^T$, where $y_1,\ldots,y_r$ are the homogeneous coordinates on ${\mathbf P}(U)$
corresponding to the basis $B$. 

\begin{cor}
The configuration incidence variety $W_U$ satisfies 
\begin{equation}\label{eq_cor}
{\rm lct}\big(V^*\times {\mathbf P}(U), W_U)=r.
\end{equation}
In particular, $W_U$ is a local complete intersection of pure codimension $r$. 
If, moreover, the matroid $M_U$ is connected, then the configuration incidence variety $W_U$ has rational singularities.
\end{cor}

\begin{proof}
Since the polynomial in (\ref{eq_Patterson}) is clearly square-free (in the sense that it is of degree $\leq 1$ with respect to each variable),
it follows from \cite[Lemma~3.3]{BMW} that ${\rm lct}(V^*,Z_U)=1$, and we deduce the formula in (\ref{eq_cor}) from Theorem~\ref{thm1_main}.
The fact that $W_U$ is a local complete intersection of pure codimension $r$ then follows from Proposition~\ref{max_lct}.

If we assume that the matroid $M_U$ is connected, then it follows easily that the polynomial in (\ref{eq_Patterson}) is irreducible (see, for example \cite[Corollary~2.19]{BW}
for a more general statement). Using again the fact that the polynomial is square-free, we deduce from \cite[Theorem~1.1]{BMW} that the hypersurface $Z_U$ has rational singularities. The fact that $W_U$ has rational singularities now follows from
Theorem~\ref{thm2_main}. 
\end{proof}

\begin{rmk}
We may also consider the subscheme $\widetilde{W}_U$ of $V^*\times U$ given by the same equations as $W_U$. It follows from Remark~\ref{rmk_affine_version}
that this is a local complete intersection of pure codimension $r$ in $V^*\times U$ and ${\rm lct}(V^*\times U, \widetilde{W}_U)=r$. However, $\widetilde{W}_U$ never has
rational singularities.
\end{rmk}

\begin{rmk} \label{rmk:charp}
It is shown in \cite{BDSW} that if $M_U$ is connected, then $W_U$ has $F$-rational type (which implies it has rational singularities) and, in fact, it is strongly $F$-regular when 
the configuration is defined over an $F$-finite field of positive characteristic. The argument in loc. cit. utilizes the matroid-theoretic structure in an essential way, hence it does not apply to arbitrary $Z_A$ and $W_A$.
\end{rmk}

\begin{rmk} \label{rmk:ConfigsNotGeneric}
    Most configuration hypersurfaces $Z_U$ do not arise from determinants of $1$-generic Patterson matrices: it is easy to see that a Patterson matrix is $1$-generic if and and only if $U \subseteq V$ 
    contains no two vectors $v$, $w$ whose Hadamard product vanishes.
\end{rmk}

\end{document}